\documentclass[12pt]{amsart}

\usepackage{amssymb}

\usepackage{enumitem}
\usepackage{tikz}
\usepackage{tikz-cd}
\usetikzlibrary{matrix}
\usepackage{graphicx}
\usepackage{xcolor}
\usepackage{comment}
\makeatletter
\@namedef{subjclassname@2010}{%
  \textup{2020} Mathematics Subject Classification}
\makeatother

\usepackage[T1]{fontenc}

\usepackage{amsmath,amssymb,amsthm,amsfonts,latexsym,amscd,hyperref}
\allowdisplaybreaks
\hypersetup{colorlinks,linkcolor={blue},citecolor={blue},urlcolor={blue}}
\usepackage{xcolor}
\usepackage{mathtools}
\usepackage{leftindex}
\newtheorem{thm}{Theorem}[section]
\newtheorem{cor}[thm]{Corollary}
\newtheorem{lem}[thm]{Lemma}
\newtheorem{prop}[thm]{Proposition}
\theoremstyle{definition}
\newtheorem{defn}[thm]{Definition}
\theoremstyle{remark}
\newtheorem{rem}[thm]{Remark}
\theoremstyle{example}
\newtheorem{exm}[thm]{Example}
\theoremstyle{observation}

\theoremstyle{application}

\theoremstyle{note}

\numberwithin{equation}{section}
\newcommand{\vertiii}[1]{{\left\vert\kern-0.25ex\left\vert\kern-0.25ex\left\vert #1
    \right\vert\kern-0.25ex\right\vert\kern-0.25ex\right\vert}}

\begin{document}


\baselineskip=17pt


\title[A generalization of Lipschitz mappings]{A generalization of Lipschitz mappings}
\author[A. K. Karn]{Anil Kumar Karn}
\address{Anil Kumar Karn, School of Mathematical Sciences, National Institute of Science Education and Research Bhubaneswar, An OCC of Homi Bhabha National Institute, P.O. Jatni, Khurda, Odisha 752050, India.}
\email{anilkarn@niser.ac.in}
\author[A. Mandal]{Arindam Mandal$^\dagger$}
\address{Arindam Mandal, School of Mathematical Sciences, National Institute of Science Education and Research Bhubaneswar, An OCC of Homi Bhabha National Institute, P.O. Jatni, Khurda, Odisha 752050, India.}
\email{arindam.mandal@niser.ac.in}
\thanks{Second author is supported by a research fellowship from the Department of Atomic Energy (DAE), Government of India.}
\thanks{$\dagger$ \tt{Corresponding author}}
\date{}

\begin{abstract}
	Using the notion of modulus of continuity at a point of a mapping between metric spaces, we introduce the notion of extensively bounded mappings generalizing that of Lipschitz mappings. We also introduce a metric on it which becomes a norm if the codomain is a normed linear space. We study its basic properties. We also discuss a linearization of an extensively bounded mapping into a bounded linear mapping. As an application, we introduce the notion of extensively bounded operator ideals. We also discuss extensively bounded finite rank and extensively bounded compact mappings and their corresponding operator ideals.
\end{abstract}

\subjclass[2020]{Primary 46B28; Secondary 46B45, 47L20}
\keywords{Modulus of continuity at a point, Lipschitz mappings, extensively bounded mappings, dilation.}

\maketitle
\pagestyle{myheadings}
\markboth{A. K. Karn and A. Mandal}{A generalization of Lipschitz mappings}

\section{Introduction} 

Let $M$ and $N$ be metric spaces. A mapping $f: M \to N$ is said to be \emph{Lipschitz}, if there exists $k > 0$ in $\mathbb{R}$ such that $d(f(x_1), f(x_2)) \le k d(x_1, x_2)$ for all $x_1, x_2 \in M$. The set of all Lipschitz mappings from $M$ into $N$ is denoted by $Lip(M, N)$. If we replace $N$ by a normed linear space $Y$, then $Lip(M, Y)$ is a vector space. In fact, $Lip^b(M, Y)$ (the subspace of all bounded mappings in $Lip(M, Y)$) is a normed linear space under a suitable norm (described later). In the literature, Lipschitz mappings between a pair of Banach spaces which takes $0$ to $0$ have been considered as a natural non-linear generalization of bounded linear maps between Banach spaces. Such mappings and the corresponding operator spaces and related operator ideals have been studied from several points of view in recent years in \cite{LFBS}, \cite{SOLAHFATA}, \cite{APITLM}, \cite{LCO}, \cite{LPCM}, \cite{OSBILFS}, \cite{LOITAP}, \cite{LPSO}, \cite{IRASOFOLS}, \cite{LOHLF}, \cite{OLOI}.  For the basics on this part, we refer to \cite{LA}. 

In this paper, we propose to expand the scope and develop a similar theory in a broader setup. The main tool for this generalization is the notion of \emph{modulus of continuity} (see \cite{GNFA}) of a mapping between metric spaces. This notion was formally introduced by H. Lebesgue in 1910 for the functions of real variables. Using the modulus of continuity, we can determine Lipschitz mappings. If we consider a modification in the modulus of continuity to make it point specific, it leads to a generalization of Lipschitz mappings making them point specific as well. If the metric spaces have fixed designated points and if a pointwise Lipschitz mapping takes the designated point in the domain to the designated point of the codomain, we call it \emph{extensively bounded}. 

It may be recalled that the space of Lipschitz mappings between metric spaces (which take the designated point in the domain to the designated point of the codomain) do not have any metric defined in it. However, in the space of extensively bounded mappings, we have been able to introduce a metric. Besides, the set of extensively mappings from a metric space into a normed linear space again form a vector space and we can introduce a norm in it as well. Moreover, if we consider the composition of two suitable extensively bounded mappings, it results in another extensively bounded mapping. In this sense we obtain a `multiplication' between two extensively bounded mappings defined on a normed space. Even though this multiplication is not bilinear, still we get an operator algebra or more generally, an operator ideal like structure. In this sense, we have introduced a category bigger than that of the spaces of Lipschitz mappings between pairs of Banach spaces.  From this paper onwards, we shall initiate a study of such mappings and spaces from various points of view. 

Other than introduction in Section $1$, the paper is divided in remaining four more sections. In Section $2$, we consider the notion of \emph{modulus of continuity at a point} of a mapping between metric spaces. As a consequence, we introduce the notion of \emph{Lipschitz condition at a point} and obtain its equivalence in the usual sense of Lipschitz condition. Next, we introduce a semi-norm and consequently, a norm on the space of all the mappings from a metric space into a normed linear space which are Lipschitz at a designated point in the domain. We prove that this space is linearly isometric to an $\ell_{\infty}$-space. We characterize a bounded Lipschitz mapping in terms of pointwise Lipschitz mappings. 

In Section $3$, we discuss a generalization of the space of all Lipschitz mappings between two metric spaces with designated points which maps the designated point to the designated point and introduce the notion of \emph{extensively bounded mappings} between two metric spaces with designated points which again maps the designated point to the designated point. We study a metric and completeness in the space of extensively bounded mappings. We show that when the codomain is replaced by a normed linear space, the space of extensively bounded mappings also becomes so. We also discuss its relation with the Lipschitz mappings that vanish at the designated point. 

In Section $4$, we first describe a Banach space $F_e(M)$ which is a predual of the Banach space $M^e$ of real valued extensively bounded mappings on a metric space $M$. This structure enable us to discuss a linearization of normed linear space valued extensively bounded mappings. We discuss the position of $X^*$ in $X^e$. Next, we describe a (non-bilinear) multiplication between a pair of extensively bounded mappings between suitable normed linear spaces and a corresponding (non-linear) morphism.

In Section $5$, as application of the theory developed in the previous sections, we propose to establish a theory of \emph{extensively bounded operator ideals}. In addition to this, we describe two basic examples. We introduce the notion of (non-linear) finite rank extensively bounded mappings and that of extensively compact mappings. We also discuss the corresponding extensively bounded operator ideals.
 
\section{Modulus of continuity at a point}

Let $M$ and $N$ be metric spaces and let $f: M \to N$. For $t \in \mathbb{R}$ with $t > 0$, we define 
$$\omega_f(t) := \sup \lbrace d(f(x), f(y)): d(x, y) \le t \rbrace.$$
Then $\omega_f: (0, \infty) \to [0, \infty]$ and is called \emph{modulus of continuity of $f$}.  Note that $\omega_f$ is an increasing function. It is also well known (see, for example, \cite[chapter 1]{GNFA}) that 
\begin{itemize}
	\item $f$ is uniformly continuous in $M$ if and only if $\lim\limits_{t \to 0^+} \omega_f(t) = 0$; and 
	\item $f$ is Lipschitz in $M$ if and only if $\sup \lbrace \frac{\omega_f(t)}{t}: t > 0 \rbrace < \infty$.
\end{itemize}
The modulus of continuity does not describe non-uniformly continuous mappings directly. But it can be done with a slight modification. We fix $x_0 \in M$ and for $t \in \mathbb{R}$ with $t > 0$, we define 
$$\omega_f^{x_0}(t) := \sup \lbrace d(f(x), f(x_0)): d(x, x_0) \le t \rbrace.$$ 
Then $\omega_f^{x_0}: (0, \infty) \to [0, \infty]$ again is an increasing function and $f$ is continuous at $x_0$ if and only if $\lim\limits_{t \to 0^+} \omega_f^{x_0}(t) = 0$. We shall call $\omega_f^{x_0}$ the \emph{modulus of continuity of $f$ at $x_0$}. 

Imitating the modulus of continuity, we consider the following notion introduced by Messerschmidt, Miek in \cite[Definition 2.6]{APLST}.
\begin{defn}
	Let $M$ and $N$ be metric spaces and let $f: M \to N$. For a fixed $x_0 \in M$, we say that $f$ is \emph{Lipschitz at $x_0$}, if 
	$$\sup \left\lbrace \frac{\omega_f^{x_0}(t)}{t}: t > 0 \right\rbrace < \infty.$$ 
	The set of all functions $f: M \to N$ which are Lipschitz at $x_0$ is denoted by $Lip^{x_0}(M, N)$. 
\end{defn}
\begin{prop}
	Let $M$ and $N$ be metric spaces and let $f: M \to N$. Then for any $x_0 \in M$, we have $f \in Lip^{x_0}(M, N)$ if and only if there exists $k > 0$ such that $d(f(x), f(x_0)) \le k d(x, x_0)$ for all $x \in M$.
\end{prop} 
\begin{proof}
	First, we assume that $f$ is Lipschitz at $x_0$ for some $x_0 \in M$. Then $k := \sup \left\lbrace \frac{\omega_f^{x_0}(t)}{t}: t > 0 \right\rbrace < \infty$. Let $x \in M \setminus \{x_{0}\}$ and put $t = d(x,x_0) > 0$. Then by definition, $d(f(x), f(x_0)) \leq \omega_{f}^{x_{0}}(t)$ so that  
	$$\frac{d(f(x), f(x_0))}{d(x, x_0)} \leq \frac{\omega_{f}^{x_0}(t)}{t} \leq k.$$ 
	Thus $d(f(x), f(x_0)) \le k d(x, x_0)$ for all $x \in M$. 
	
	Conversely, we assume that $d(f(x), f(x_0)) \le k d(x, x_0)$ for all $x \in M$. We show that $\sup \left\lbrace \frac{\omega_f^{x_0}(t)}{t}: t > 0 \right\rbrace \le k$. Let $t > 0$ and $\epsilon > 0$. Then we can find $x \in M$, $d(x, x_0) \le t$ such that $d(f(x), f(x_0)) > \omega_f^{x_0}(t) - \epsilon t$. Thus 
	$$\omega_f^{x_0}(t) - \epsilon t < d(f(x), f(x_0)) \le k d(x, x_0) \le k t.$$ 
	Now it follows that $\frac{\omega_f^{x_0}(t)}{t} - \epsilon < k$ for all $t > 0$ and $\epsilon > 0$. Thus $\sup \left\lbrace \frac{\omega_f^{x_0}(t)}{t}: t > 0 \right\rbrace \le k < \infty$. 
\end{proof}
The above result establishes that the two notions, namely pointwise $\alpha$-Lipschitz and strongly pointwise $\alpha$-Lipschitz, considered by Messerschmidt, Miek in \cite[Definition 2.6]{APLST} are actually equivalent.  

We observe that the notion of the Lipschitz condition at a point is different from that of \emph{local Lipschitz at a point}. ($f$ is said to be \emph{locally Lipschitz at a point $x_0$}, if there exists a neighbourhood $N_{x_0}$ of $x_0$ in $M$ such that $f$ is Lipschitz in $N_{x_0}$ (see, for example, \cite[Chapter 2]{LF}).) 

We say that $f$ is \emph{pointwise Lipschitz in $M$}, if $f \in Lip^x(M, N)$ for each $x \in M$. The set of all such mappings is denoted by $Lip^P(M, Y)$. Thus $Lip^P(M, Y) = \bigcap\limits_{x \in M} Lip^x(M, N)$. 

If we replace $N$ by a (real) normed linear space $Y$, then $Lip^{x_0}(M, Y)$ is also a (real) vector space under pointwise operations. Moreover, the value $\Vert f \Vert_{x_0} := \sup \left\lbrace \frac{\omega_f^{x_0}(t)}{t}: t > 0 \right\rbrace$ for all $f \in Lip^{x_0}(M, Y)$ determines a semi-norm $\Vert\cdot\Vert_{x_0}$ on $Lip^{x_0}(M, Y)$. In fact, we have 
\begin{eqnarray*}
	\Vert f \Vert_{x_0} &=& \inf \lbrace k > 0: \Vert f(x) - f(x_0) \Vert \le k d(x, x_0) ~ \mbox{for all} ~ x \in M \rbrace \\ 
	&=& \sup \left\lbrace \frac{\Vert f(x) - f(x_0) \Vert}{d(x, x_0)}: x \in M, x \ne x_0 \right\rbrace  
\end{eqnarray*}
for all $f \in Lip^{x_0}(M, Y)$. Also $\Vert f \Vert_{x_0} = 0$ if and only if $f$ is a constant function. Thus for $f \in Lip^{x_0}(M, Y)$ the value 
$$\Vert f \Vert_{L^{x_0}} := \max \lbrace \Vert f \Vert_{x_0}, \Vert f(x_0) \Vert \rbrace$$ 
determines a norm on $Lip^{x_0}(M, Y)$. We record this fact and some more in the next result. 
\begin{prop}\label{Lip0p}
	Let $M$ be a metric space and $Y$ a normed linear space. Then $(Lip^x(M, Y), \|.\|_{L^{x}})$ is isometrically isomorphic to $\ell_{\infty}(M, Y)$. Thus $(Lip^x(M, Y), \|.\|_{L^{x}})$ is a Banach space, if and only if $Y$ is so. 
\end{prop}
\begin{proof}
	Let $f \in \ell_{\infty}(M, Y)$. Define $\Phi_x(f): M \to Y$ given by 
	$$\Phi_x(f)(x') = d(x', x) f(x') + f(x)$$ 
	for all $x' \in M$. Then 
	$$\Vert \Phi_x(f)(x') - \Phi_x(f)(x) \Vert = \Vert f(x') \Vert d(x', x) \le \Vert f \Vert_{\infty} d(x', x)$$ 
	for all $x' \in M$. Thus $\Phi_x(f) \in Lip^x(M, Y)$ for all $f \in \ell_{\infty}(M, Y)$. Now, it follows that $\Phi_x: \ell_{\infty}(M, Y) \to (Lip^x(M, Y), \|.\|_{L^{x}})$ is a linear mapping. Further 
	\begin{eqnarray*}
		\Vert \Phi_x(f) \Vert_{L^x} &=& \sup \left\lbrace\frac{\Vert \Phi_x(f)(x') - \Phi_x(f)(x) \Vert}{d(x', x)}, \Vert f(x) \Vert: x' \in M \setminus \lbrace x \rbrace  \right\rbrace \\ 
		&=& \sup \lbrace \Vert f(x') \Vert: x' \in M \rbrace \\ 
		&=& \Vert f \Vert_{\infty}
	\end{eqnarray*} 
	for all $f \in \ell_{\infty}(M, Y)$ so that $\Phi_x$ is a linear isometry. Finally, if $g \in Lip^x(M, Y)$, we set 
	$$h(x') = \begin{cases}
		\frac{g(x') - g(x)}{d(x', x)}, &~ \mbox{if} ~ x' \ne x \\ 
		g(x), &~ \mbox{if} ~ x' = x.
	\end{cases}$$
	As $g \in Lip^x(M, Y)$, there exists $k > 0$ such that $\Vert g(x') - g(x) \Vert \le k d(x', x)$ for all $x' \in M$. Thus $\sup \lbrace \Vert h(x') \Vert: x' \in M \rbrace \le k$ so that $h \in \ell_{\infty}(M, Y)$. Also, then $\Phi_x(h) = g$ so that $\Phi_x$ is the required isometry.
\end{proof} 
Let $Lip(M, Y)$ denote the set of all Lipschitz mappings from $M$ into $Y$ where $M$ is a metric space and $Y$ is a normed linear space. Also then $Lip^b(M, N)$ denotes the set of bounded mappings in $Lip(M, N)$. Let us recall that for $f \in Lip^b(M, Y)$ the value 
$$\Vert f \Vert_L := \sup \left\lbrace\frac{\Vert f(x_1) - f(x_2) \Vert}{d(x_1, x_2)}, \Vert f \Vert_{\infty}: x_1, x_2 \in M ~ \mbox{with} ~ x_1 \ne x_2  \right\rbrace$$ 
determines a norm $\Vert\cdot\Vert_L$ on $Lip^b(M, Y)$.(see \cite[Chapter 2, pp-35]{LA}) Further, $(Lip^b(M, Y), \Vert\cdot\Vert_L)$ is a Banach space whenever $Y$ is so. We note that $(Lip^x(M, Y), \|.\|_{L^{x}})$ generalizes $(Lip^b(M, Y), \Vert\cdot\Vert_L)$ in the following sense. 

Consider $f \in \bigcap\limits_{x \in M} Lip^x(M, Y)$. Then 
\begin{eqnarray*}
	\sup\limits_{x \in M} \Vert f \Vert_{L^x} &=& \sup\limits_{x \in M} \sup \left\lbrace \frac{\Vert f(x') - f(x) \Vert}{d(x', x)}, \Vert f(x) \Vert: x' \in M \setminus \lbrace x \rbrace \right\rbrace \\ 
	&=& \max \left\lbrace \sup\limits_{x, x' \in M; x \ne x'} \left(\frac{\Vert f(x') - f(x) \Vert}{d(x', x)} \right), \sup\limits_{x \in M} \Vert f(x) \Vert \right\rbrace 
\end{eqnarray*}
Thus, if $f$ is bounded in $M$ with $\sup\limits_{x \in M} \Vert f \Vert_x < \infty$, then $f \in Lip^b(M, Y)$. The converse may be verified in a routine way. We summarize the observation in the following result. 
\begin{prop}\label{Lip0}
	Let $M$ be a metric space $M$ and $Y$ a normed linear space. Then $f \in Lip^b(M, Y)$ if and only if $f$ is bounded in $M$ and $f \in \bigcap\limits_{x \in M} Lip^x(M, Y)$ with $\sup\limits_{x \in M} \Vert f \Vert_x < \infty$. 
\end{prop} 
\begin{rem}
	Since $(Lip^x(M, Y),\|.\|_{L^{x}})$ is isometrically isomorphic to $\ell_{\infty}(M, Y)$ for any $x \in M$, we get that $(Lip^{x_1}(M, Y), \Vert\cdot\Vert_{L^{x_1}})$ is isometrically isomorphic to $(Lip^{x_2}(M, Y), \Vert\cdot\Vert_{L^{x_2}})$ for all $x_1, x_2 \in M$. 
\end{rem}
\begin{prop}
	Let $X$ and $Y$ be normed linear spaces. Assume that $x_1, x_2 \in X$ with $x_1 \ne x_2$ and consider the surjective isometry 
	$$\Phi_{x_2} \circ \Phi_{x_1}^{-1}: (Lip^{x_1}(M, Y), \Vert\cdot\Vert_{L^{x_1}}) \to (Lip^{x_2}(M, Y), \Vert\cdot\Vert_{L^{x_2}}).$$ 
	If $f \in Lip^{x_1}(X, Y) \bigcap Lip^{x_2}(X, Y)$ with $\Phi_{x_2} \circ \Phi_{x_1}^{-1}(f) = f$, then $f(x) = 0$ for all $x \in X$ with $\Vert x - x_1 \Vert \ne \Vert x - x_2 \Vert$. In particular, $f(x_1) = 0 = f(x_2)$. 
\end{prop}
\begin{proof}
	First we note that 
	$$\Phi_{x_2} \circ \Phi_{x_1}^{-1}(f)(x) = 
	\begin{cases} \frac{\Vert x - x_2 \Vert}{\Vert x - x _1 \Vert} (f(x) - f(x_1)) + \frac{f(x_2) - f(x_1)}{\Vert x_2 - x_1 \Vert}, ~ &\mbox{if} ~ x \ne x_1 \\ 
		\Vert x_1 - x_2 \Vert f(x_1) +  \frac{f(x_2) - f(x_1)}{\Vert x_2 - x_1 \Vert}, ~ &\mbox{if} ~ x = x_1. 
	\end{cases}$$ 
	Thus as $f(x) = \Phi_{x_2} \circ \Phi_{x_1}^{-1}(f)(x)$, we get $f(x_2) = \frac{f(x_2) - f(x_1)}{\Vert x_1 - x_2 \Vert}$ and hence $f(x_2) = \Vert x_1 - x_2 \Vert f(x_2) + f(x_1)$. Also $f(x_1) = \Vert x_1 - x_2 \Vert f(x_1) + f(x_2)$. 
	Now it follows that 
	$$f(x_1) = - f(x_2) = \frac{f(x_1) - f(x_2)}{\Vert x_1 - x_2 \Vert}.$$
	Also, on simplifying, we get 
	$$\Vert x - x_1 \Vert (f(x) + f(x_1)) = \Vert x - x_2 \Vert (f(x) - f(x_1))$$ 
	or, equivalently,
	$$(\Vert x - x_2 \Vert - \Vert x - x_1 \Vert) f(x) = (\Vert x - x_2 \Vert + \Vert x - x_1 \Vert) f(x_1)$$ 
	for all $x \in X$. Put $x_0 = \frac 12 (x_1 + x_2)$. Then $\Vert x_0 - x_1 \Vert = \Vert x_0 - x_2 \Vert = \frac 12 \Vert x_1 - x_2 \Vert > 0$. Thus $f(x_1) = 0$ as $\Vert x_0 - x_2 \Vert + \Vert x_0 - x_1 \Vert = \Vert x_1 - x_2 \Vert > 0$ and consequently, $f(x_2) = 0$ too. Therefore, $(\Vert x - x_2 \Vert - \Vert x - x_1 \Vert) f(x) = 0$ for all $x \in X$. Hence the result follows. 
\end{proof}

\section{extensively bounded mappings}

Let $X$ and $Y$ be normed linear spaces and consider 
$$Lip_0(X, Y) := \lbrace f \in Lip(X, Y): f(0) = 0 \rbrace.$$ 
We recall that this space has been studied in the literature (see, for example, \cite[Chapter 1]{LA}) as a non-linear generalization of bounded linear mappings between the normed linear spaces. We propose a further generalization and consider a particular (closed) subspace of $Lip^x(M, Y)$ as under: 
$$Lip_x^x(M, Y) := \lbrace f \in Lip^x(M, Y): f(x) = 0 \rbrace.$$ 
For $f \in Lip_x^x(M, Y)$, we have 
$$\Vert f \Vert_{L^x} = \Vert f \Vert_x = \sup \left\lbrace\frac{\Vert f(x') \Vert}{d(x', x)}: x' \in M \setminus \lbrace x \rbrace  \right\rbrace.$$ 
Thus $(Lip_x^x(M, Y), \Vert\cdot\Vert_x)$ is isometrically isomorphic to $\ell_{\infty}(M \setminus \lbrace x \rbrace, Y)$ and $(Lip_x^x(M, Y), \Vert\cdot\Vert_x)$ is a Banach space, if and only if $Y$ is so. Note that $f: M \to Y$ is in $Lip_x^x(M, Y)$ if and only if there exists $k > 0$ such that $\Vert f(x') \Vert \le k d(x', x)$ for all $x' \in M$. Thus  
$$\Vert f \Vert_x = \inf \lbrace k > 0: \Vert f(x') \Vert \le k d(x', x) ~ \mbox{for all} ~ x' \in M \rbrace.$$ 

Let us fix $x \in M$ and consider the space 
$$Lip_x(M, Y) := \lbrace f: M \to Y| f ~ \mbox{is Lipschitz in} ~ M ~ \mbox{and} ~ f(x) = 0 \rbrace$$ 
together with the norm $Lip(\cdot)$ given by 
\begin{eqnarray*}
	Lip(f) &:=& \inf \lbrace k > 0: \Vert f(x_1) - f(x_2) \Vert \le k d(x_1, x_2) ~ \mbox{for all} ~ x_1, x_2 \in M \rbrace \\
	&=& \sup \left\lbrace \frac{\Vert f(x_1) - f(x_2) \Vert}{d(x_1, x_2)}: x_1, x_2 \in M ~ \mbox{with} ~ x_1 \ne x_2 \right\rbrace
\end{eqnarray*} 
for all $f \in Lip_x(M, Y)$. Then $(Lip_x(M, Y), Lip(\cdot))$ is a normed linear space and is complete whenever $Y$ is so.( see \cite[Chapter 8, Theorem 8.1.3]{LF} ) Thus $(Lip_x^x(M, Y), \Vert\cdot\Vert_x)$ is a natural generalization of $(Lip_x(M, Y), Lip(\cdot))$. 

Being isometric to $\ell_{\infty}(M \setminus \lbrace x \rbrace, Y)$, $(Lip_x^x(M, Y), \Vert\cdot\Vert_x)$ may appear in a poor light. But we shall show in another paper that when $M$ is also replaced by a normed linear space $X$, $(Lip_x^x(X, Y), \Vert\cdot\Vert_x)$ assumes a non-commutative structure which is not known in $\ell_{\infty}$ spaces. Further, this facet leads us to an operator ideal theory generalizing that of the Lipschitz operators. Moreover, we shall later infuse continuity to make it less general. 

Towards this goal, we re-describe $(Lip_x^x(M, \mathbb{R})$)  with a simpler notation. Since for any $x_1, x_2 \in M$, $(Lip_{x_1}^{x_1}(M, Y), \Vert\cdot\Vert_{x_1})$ is isometrically isomorphic to $(Lip_{x_2}^{x_2}(M, Y), \Vert\cdot\Vert_{x_1})$, there is no loss of generality in fixing a designated point. We generally describe it as $0$. 
\begin{defn}
	Let $M, N$ be metric spaces with $0$ as the distinguished point. A mapping $f: M \to N$ is said to be \emph{extensively bounded}, if there exists $k > 0$ such that $d( f(x),0) \le k d(x, 0)$ for all $x \in M$. The set of all extensively bounded mappings from $M$ to $N$ is denoted as $E(M, N)$. 
\end{defn}
\begin{thm} \label{E.S  is B.sp}
    Let $M$ and $N$ be metric spaces with $0$ as distinguished points. 
    \begin{enumerate}
        \item The map $d_{e}: E(M, N) \times E(M, N) \xrightarrow{} [0,\infty)$ given by $$d_{e}(T, S) := \sup\limits_{x \neq 0} \frac{d(T(x), S(x))}{d(x, 0)},$$
        for all $T, S \in E(M, N)$ determines a metric on $E(M, N)$.
        \item $(E(M, N), d_{e})$ is a complete metric space whenever $N$ is so.
        \item If $N$ replaced by a (real) normed linear space $Y$, then $E(M, Y)$ becomes a vector space and $d_{e}$ induces a norm $\Vert\cdot\Vert_e$ on $E(M, Y)$. Consequently, $(E(M, Y), \Vert\cdot\Vert_e)$ is a Banach space whenever $Y$ is so. 
    \end{enumerate}
\end{thm}
\begin{proof}
  Let $f, g \in E(M, N)$. Then there exist $k > 0$ and $l > 0$ such that $d(f(x), 0) \le k d(x, 0)$ and $d(g(x), 0) \le l d(x, 0)$ for all $x \in M$. Now using the triangle inequality, we get 
  $$d(f(x), g(x)) \leq d(f(x), 0)+d(g(x), 0) \le (k + l) d(x, 0)$$ 
  for any $x \in M$.
    Thus for any $x \in M \setminus\{0\}$, we have 
    $$\frac{d(f(x), g(x))}{d(x, 0)} \leq \frac{d(f(x), 0)}{d(x, 0)} + \frac{d(g(x), 0)}{d(x, 0)} \le k + l.$$ 
     Therefore, $d_{e}$ is well-defined. It follows from the definition that $d_{e}(f, g) = d_{e}(g, f)$ for any $f, g \in E(M,N)$.

    Assume that $d_{e}(f, g) = 0$ for some $f, g \in E(M,N)$. Then for all $x \in M\setminus \{0\}$, we have $d(f(x), g(x)) = 0$ . Also $f(0) = 0 = g(0)$ so that $f = g$. Thus $d_{e}$ is a metric on $E(M, N)$.

  Let $(f_{n})$ be a Cauchy sequence in $(E(M, N), d_{e}).$ Then given $\epsilon > 0$, we can find $n_0 \in \mathbb{N}$ such that $d_{e}( f_m, f_n) \le \epsilon$ for all $m, n \ge n_0$. In other words, $d(f_m(x) , f_n(x)) \le \epsilon d(x, 0)$ for all $x \in M$ if $m, n \geq n_0$. In particular, for each $x \in M$, $(f_{n}(x))$ is Cauchy in $N$. Since $N$ is complete, $(f(x_n))$ converges to some $y \in N$. As $y$ is determined by $x \in M$ uniquely, we get a map $f: M \to N$ such that $f_{n}$ converges to $f$ point-wise. Also then 
	$$d( f(x) , f_n(x)) = \lim\limits_{m \to \infty} d( f_{m}(x) , f_{n_0}(x)) \le \epsilon d(x, 0)$$ 
	for all $x \in M$, if $n \ge n_0$. Therefore, 
	$$d(f(x),0)  \le d(f(x), f_{n_0}(x))  + d( f_{n_0}(x),0)  \le (\epsilon + d(f_{n_0},0)) d(x, 0)$$ 
	for all $x \in M$. Thus $f \in E(M, N)$. Also then 
	$$d_{e} (f, f_n) = \sup\limits_{x \neq 0, x \in M} \frac{d( f(x), f_n(x))}{d(x,0)} \le \epsilon$$ 
	for all $n \geq n_0$. Thus $(E(M, N), d_{e})$ is complete.

 Now consider $Y$ to be a (real) normed linear space. It is routine to check that $E(M, Y)$ is a (real) vector space under pointwise addition and scalar multiplication. Again, for $f, g \in E(M, Y)$, 
 $$d_{e}(f, g) = \sup\limits_{x \neq 0} \frac{\Vert f(x)- g(x) \Vert}{d(x, 0)} =d_{e}(f - g, 0).$$  
 Further, for $\alpha \in \mathbb{K}$, 
 $$d_e(\alpha f,\alpha g)=\sup\limits_{x \neq 0} \frac{\|\alpha f(x)- \alpha g(x)\|}{d(x, 0)}= |\alpha| \sup\limits_{x \neq 0} \frac{\Vert f(x)- g(x) \Vert}{d(x, 0)} = \vert \alpha \vert d_e(f, g).$$ 
 Thus the metric $d_{e}$ induces a norm  $\Vert\cdot\Vert_e$ on $E(M, Y)$ given by 
 $$\Vert f \Vert_e = d_e(f, 0)$$
 for all $f \in E(M, Y)$ whenever $Y$ is a (real) normed linear space.

 In particular, if $Y$ is a Banach space, then  $(E(M, Y), \Vert\cdot\Vert_e)$ is also a Banach space.
\end{proof}
\begin{rem}
   In general, there is no specific metric defined on $Lip_{0}(M,N)$. However, being a subspace of $E(M, N)$, we can provide the restriction of $d_e$ on it. But we do not know whether, in general, $Lip_0(M, N)$ is complete with this metric or not, even assuming the completeness of $N$.
\end{rem}
\begin{exm}
	In general, $E(M, N)$ and $C(M, N)$ are distinct spaces. For example, let $M = N = \mathbb{R}$ and consider the functions 
		\[
		g(x) =
		\begin{cases}
			2 &\text{if } x =1 \\
			x &\text{if } x \neq 1

		\end{cases}
		\]
		and $h(x)= x^{2}$. Then $g \in E(\mathbb{R}, \mathbb{R}) \setminus C(\mathbb{R}, \mathbb{R}) $ and $ h \in C(\mathbb{R}, \mathbb{R}) \setminus E(\mathbb{R}, \mathbb{R}).$
\end{exm}
 Now we consider the space $E_c(M, Y) = C(M, Y) \cap E(M, Y)$ as a normed subspace of $E(M, Y)$. 

 \begin{prop}
     Consider $E_c(M, N) := C(M, N) \cap E(M, N)$ where $M$ and $N$ are metric spaces with distinguished points as $0$.  Then $E_c(M, N)$ is a closed subspace of $E(M, N)$ and hence is complete whenever $N$ is so.
 \end{prop}
 \begin{proof}
     Let $(f_{n})$ be a sequence in $E_c(M, N)$ such that it converges to $f \in E(M, N)$. We show that $f$ is continuous.
     
     {\bf Step I.} Continuity of $f$ at $0 \in M.$
     
     Since $f_{n} \to f$, given $\epsilon > 0$, there exists $n_{0} \in \mathbb{N}$ such that $d_{e}( f_n , f ) < \epsilon$ for all $n \ge n_0$. Thus $d(f_{n}(x), f(x)) \leq \epsilon d(x, 0)$ for all $n \geq n_{0}$ and $x \in M$. Now $f_{n_{0}}$ is continuous at $0$ so there exists $\delta > 0$ such that $d(f_{n_{0}}(x),0) < \frac{\epsilon}{2}$ whenever $d(x, 0) < \delta$. (In fact, $f_{n_0}(0) = 0$.) Consider $\delta^{'} = min\{\delta, \frac{1}{2}\}> 0$, then for $d(x, 0) < \delta^{'}$, we have 
     \begin{eqnarray*}
         d(f(x), f(0)) = d( f(x),0) &\leq& d (f_{n_{0}}(x), f(x)) +  d(f_{n_{0}}(x),0) \\
         &<& \epsilon \ d(x, 0) + \frac{\epsilon}{2}\\
         &\le& \frac{\epsilon}{2} + \frac{\epsilon}{2}= \epsilon.
     \end{eqnarray*}
     Thus $f$ is continuous at $0.$
     
     {\bf Step II.} Continuity of $f$ at $x_{0} \in M\setminus \{0\}$. 
     
      Since $f_{n} \to f$, given $\epsilon > 0$, we can find $n_1 \in \mathbb{N}$ such that 
      $$ d(f_{n}(x), f(x)) \le \frac{\epsilon}{4 ~ d(x_0, 0)} d(x, 0)$$ 
      for all $n \geq n_1$ and $x \in M$. Also, as $f_{n_{0}}$ is continuous at $x_{0}$, there exists $\delta > 0$ such that $ d(f_{n_1}(x), f_{n_1}(x_0)) < \frac{\epsilon}{4},$ whenever $d(x, x_{0}) < \delta$. Choose $\delta^{'} = min\{d(x_{0},0), \delta\} > 0.$ Then, for $x \in M$ with $d(x,x_{0}) < \delta^{'}$, we have 
\begin{eqnarray*}
    d (f(x),f(x_0)) &\le& d(f(x),f_{n_1}(x)) + d(f_{n_1}(x), f_{n_1}(x_0)) \\ 
    &+& d (f_{n_1}(x_0), f(x_0)) \\
    &<& \frac{\epsilon}{4\ d(x_{0},0)} d(x,0) + \frac{\epsilon}{4} + \frac{\epsilon}{4 \ d(x_{0},0)} d(x_0,0)\\
    &\leq& \frac{\epsilon \ (d(x,x_{0})+d(x_{0},0))}{4 \ d(x_{0},0)} + \frac{\epsilon}{2}\\
    &<& \frac{3 \ \epsilon}{4} + \frac{\epsilon}{4 \ d(x_{0},0)} \delta^{'} \\ 
    &\le& \frac{3 \ \epsilon}{4} + \frac{\epsilon}{4 \ d(x_{0},0)} d(x_{0}, 0) = \epsilon.
\end{eqnarray*}
	Thus $f$ is continuous in $M$, that is, $f \in E_c(M, N)$ and consequently, we conclude that 
    $E_c(M, N)$ is a closed subspace of $E(M, N)$. 
    \end{proof}
 \begin{exm} 
 	\begin{enumerate}
 		\item Consider $f : \mathbb{R} \xrightarrow{} \mathbb{R}$ such that $f(x)=xsin(x)$ then $f \in E_{c}(\mathbb{R}, \mathbb{R})$ with $\|f\|_{e}=1$, but $f$ is not uniformly continuous as $f(2n\pi + \frac{1}{n})-f(2n\pi)$ converges to $2\pi$; hence $f$ is not Lipschitz. Thus $f \in E_{c}(\mathbb{R}, \mathbb{R}) \setminus Lip_{0}(\mathbb{R}, \mathbb{R}).$
 		\item From the above example, we have that $E_{c}(\mathbb{R}, \mathbb{R})$ is not contained in the set of all uniformly continuous maps from $\mathbb{R}$ to $\mathbb{R}.$ Also, the map $f(x)=\frac{1}{1+ x^{2}} \ \forall x \in \mathbb{R}$ is uniformly continuous but not in $E_{c}(\mathbb{R}, \mathbb{R}).$ 
 	\end{enumerate}	
\end{exm} 
\begin{rem} \label{rem3.7}
\begin{enumerate}
	\item Let $M$ be a metric space with the distinguished point $0$ and $Y$ be a normed linear space, then $Lip_{0}(M, Y) \subset E_{c}(M, Y)$ with $\|f\|_{e} \leq Lip(f)$ for any $f \in Lip_{0}(M, Y) $. In fact, if $f \in Lip_{0}(M, Y)$, then $\|f(x)\| \leq Lip(f) d(x,0)$ for all $x \in M$. Thus $\|f\|_{e} \leq Lip(f)$.
    \item For any normed linear spaces $X$ and $Y$, we have 
    $$L(X, Y) \subset Lip_{0}(X, Y) \subset E_{c}(X, Y)$$ 
    with $\|T\|=Lip(T)=\|T\|_{e}$, whenever $T \in L(X, Y)$, where $L(X, Y)$ is the space of all bounded linear operators from $X$ to $Y$. To see this, let $T \in L(X, Y)$. Then 
    $$\Vert T \Vert_{e} = \sup\limits_{x \neq 0} \frac{\|T(x)\|}{\|x\|}=\|T\|$$ 
    and 
    $$Lip(T) = \sup\limits_{x \neq y} \frac{\|T(x)-T(y)\|}{\|x-y\|}= \Vert T \Vert.$$
    \end{enumerate}
\end{rem}
\begin{rem} \label{rem1.6}
	\begin{enumerate}
		\item In general, $Lip(.)$ and $\|.\|_{e}$ are not equivalent in $Lip_{0}(M, \mathbb{R}).$ To show this, we consider $M =[0,1]$ and the sequence of functions $f_{n}(x)=x^{n}$. Then $f_{n} \in Lip_{0}([0,1], \mathbb{R})$ with $\|f_{n}\|_{e}=1$ but $Lip(f_{n})=n $ for all $n \in \mathbb{N}$. 
		\item By Remark \ref{rem3.7}(1), we have $\overline{Lip_{0}(M, \mathbb{R})}^{Lip(.)} \subset \overline{Lip_{0}(M, \mathbb{R})}^{\|.\|_{e}}$. Also, by (1), the inclusion is proper.
	\end{enumerate}   
\end{rem}
\begin{exm}
    Consider the metric space $M=\{\frac{1}{n}: n \in \mathbb{N}\} \cup \{0\}$ with $0$ as the distinguished point. Then 
    $$E(M, \mathbb{R}) = \{(x_{n}) \in c_{0} : |x_{n}| \leq k.\frac{1}{n} \mbox{\ for \ all } n \in \mathbb{N}\}.$$
    \begin{proof}
        Let $f \in E(M, \mathbb{R}).$ Then $|f(\frac{1}{n})| \leq \frac{1}{n}  \|f\|_{e}$ and $\ f(0)=0$. Put $x_{n}=f(\frac{1}{n})$. Then $(x_{n}) \in c_{0}$ and $|x_{n}|\leq \|f\|_{e}. \frac{1}{n}$ for all $n \in \mathbb{N}.$ Conversely, consider $(x_{n}) \in c_{0}$ with $|x_{n}| \leq k.\frac{1}{n}$ for some $k> 0$ and all $n \in \mathbb{N}.$ Define     
$f(\frac{1}{n}) = x_{n}$ for $n \in \mathbb{N}$ and $f(0)=0$ then $f \in E(M, \mathbb{R}).$
\end{proof}
\end{exm}

\section{Dilating extensively bounded mappings}

Before we proceed further, we shall fix some notations. We denote the space $E(M, \mathbb{R})$ by $M^{e}$ and therefore, $E(E(M, \mathbb{R}), \mathbb{R})$ by $M^{ee}$ and so on. 
In this section, we shall prove that an extensively bounded mapping between metric spaces can be dilated to a Lipschitz mapping between bigger metric spaces and an extensively bounded continuous mapping from a metric space into a normed linear space can be dilated to a bounded linear mapping between (bigger) Banach spaces.
\begin{prop} \label{prop check}
  Let $M$ be a metric space with $0$ as a distinguished point. For $x \in M$ we define $\delta_{x}^{e}: M^{e} \to \mathbb{R}$ given by $\delta_{x}^{e}(f) = f(x)$ for all $f \in M^{e}$. Then $\delta_{x}^{e}$ is a bounded linear functional on $M^{e}$ with $\Vert \delta_{x}^{e} \Vert = d(x, 0)$. In general, $\|\delta_{x}^{e} - \delta_{y}^{e}\|\geq d(x, y)$ for $x, y \in M$. Next, consider the mapping $\delta_M^e: M \to (M^e)^*$ given by $\delta_M^e(x) = \delta_x^e$ for all $x \in M$. Then $\delta_M^e$ is extensively bounded. 
\end{prop}
\begin{proof}
	Clearly, $\delta_{x}^{e}$ is linear. Also for any $f \in E(M, \mathbb{R})$ with $\Vert f \Vert_e \le 1$, we have 
	$$\vert \delta_{x}^{e}(f) \vert = \vert f(x) \vert \le d(x, 0)$$ 
	so that $\delta_{x}^{e}$ is bounded with $\Vert \delta_{x}^{e} \Vert \le d(x, 0)$. Consider $f_1: M \to \mathbb{R}$ given by $f_1(z) = d(z, 0)$ for all $z \in M$. Then $f_1 \in M^{e}$ with $\Vert f_1 \Vert_e = 1$. Thus $\Vert \delta_{x}^{e} \Vert \ge \delta_{x}^{e}(f_1) = f_1(x) = d(x, 0)$ so that $\Vert \delta_{x}^{e} \Vert = d(x, 0)$ for all $x \in M$. 
	Hence, the map $\delta_{M}^{e}$ is extensively bounded. 

 Let $x, y \in M$ and define $g : M \xrightarrow{} \mathbb{R}$ given by $g(z) = d(y, z)- d(y,0)$ for all $z \in M$. Then $g \in M^{e}$ with $\|g\|_{e} \leq 1$. Therefore, $$\|\delta_{x}^{e} - \delta_{y}^{e}\|= \sup\limits_{f \in M^{e}, \ \|f\|_{e} \leq 1} |f(x)-f(y)| \geq |g(x)-g(y)| = d(x,y).$$
\end{proof} 
\begin{rem}
    Since $\delta_{x}^{e}$ is a linear map, it follows from Remark \ref{rem3.7} that $\|\delta_{x}^{e}\| = Lip(\delta_{x}^{e}) = \Vert \delta_{x}^{e} \Vert_e$.
\end{rem}
When $M$ is replaced by a normed linear space $X$, $x \mapsto \delta_{x}^{e}$ is a (non-linear) norm preserving mapping from $X$ into $(X^{e})^{\ast}$. We consider the closed subspace $F_{e}(M) = \overline{span \left\{\delta_{x}^{e} : x \in M\right\}}$ of $(M^{e})^{\ast}$. Then $\delta_{M}^{e} : M \xrightarrow{} F_{e}(M)$ given by $\delta_{M}^{e}(x)= \delta_{x}^{e}$ is a norm preserving (and therefore, extensively bounded) non-linear map such that $\overline{span\left\{\delta_M^e(M)\right\}} = F_e(M)$. Let us recall that when $M^e$ is replaced by $Lip_0(M, \mathbb{R})$, a space similar to $F_e(M)$ is called the Lipschitz-free space of $M$. As is in the case of Lipschitz functions, we show that $M^e$ is also a dual space. Actually, we prove more. 
\begin{thm}\label{lin}
    Let $M$ be a metric space and $Y$ a Banach space and consider $T \in E(M, Y)$. Then there exists a unique bounded linear map $\hat{T} : F_{e}(M) \to Y$ such that $\hat{T} \circ \delta_{M}^{e} = T$ with $\|\hat{T}\|=\|T\|_{e}$. Further, the mapping $\Phi_M^Y: E(M, Y) \xrightarrow{} L(F_{e}(M),Y)$, given by $\Phi_M^Y(T)= \hat{T}$ for $T \in E(M, Y)$, is a surjective linear isometry. We shall write simply $\Phi$ for $\Phi_M^Y$, if it is clear from the context.
\end{thm}
\begin{proof}
    Define $\hat{T} :span\{\delta_{x}^{e} : x \in M\} \xrightarrow{} Y$ by $\hat{T} \left(\sum\limits_{i=1}^{n} \alpha_{i} \delta_{x_{i}}^{e}\right)= \sum\limits_{i=1}^{n} \alpha_{i} T(x_{i})$.
    Then $\hat{T}$ is linear with $\hat{T} \circ \delta_{X}^{e} = T$. Also  
     \begin{eqnarray*}
         \left\Vert\sum\limits_{i=1}^{n} \alpha_{i} T(x_{i}) \right\Vert &=& \sup\limits_{y^{\ast} \in B_{Y^{\ast}}} \left\Vert\sum\limits_{i=1}^{n} \alpha_{i} y^{\ast}(T(x_{i})) \right\Vert \\
         &=& \|T\|_{e} \sup\limits_{y^{\ast} \in B_{Y^{\ast}} }\left\vert\sum\limits_{i=1}^{n} \alpha_{i} y^{\ast}(\frac{T}{\|T\|_{e}}(x_{i}))\right\vert \\ 
         &\leq& \|T\|_{e}\sup\left\{\left\vert \sum\limits_{i=1}^{n} \alpha_{i} f(x_{i}) \right\vert: f \in X^{e}, \|f\|_{e} \leq 1\right\} \\
         &=&\|T\|_{e} \sup\left\{\left\vert \left(\sum\limits_{i=1}^{n} \alpha_{i} \delta_{x_{i}}^{e}\right)(f)\right\vert: f \in X^{e}, \|f\|_{e} \leq 1\right\}\\
         &=& \|T\|_{e}\left\|\sum\limits_{i=1}^{n} \alpha_{i} \delta_{x_{i}}^{e}\right\|.
     \end{eqnarray*}
 Thus $\hat{T}$ is bounded and linear in $span\left\{\delta_{x}^{e} : x \in M\right\}$ with $\|\hat{T}\| \leq \|T\|_{e}$. Hence it can be extended uniquely to $F_{e}(M)$ as a bounded linear mapping, denoted again as $\hat{T}$, with $\|\hat{T}\| \leq \|T\|_{e}$.    
     
     Further,
     $$\|T\|_{e} = \sup\limits_{x \neq 0} \frac{\|Tx\|}{\|x\|} = \sup\limits_{x \neq 0} \frac{\|\hat{T}(\delta_{x}^{e})\|}{\|\delta_{x}^{e}\|} \leq \|\hat{T}\|.$$ 

     Thus $\Phi$ is a linear isometry. We prove that $\Phi$ is surjective. In fact, if $\xi \in L(F_{e}(X),Y)$, then $\xi \circ \delta_{X}^{e} \in E(X, Y)$ and $\Phi(\xi \circ \delta_{X}^{e}) = \xi$. This completes the proof.
\end{proof}
\begin{rem}
    Taking $Y = \mathbb{R}$, we note that $X^{e}$ is dual of $F_{e}(X)$.
\end{rem}
\begin{defn} \label{def e adjoint}
    For $T \in E(M, Y)$ we define $T^{e}: Y^{e} \xrightarrow{} M^{e}$ by $$T^{e}(f)=f \circ T.$$ Then $T^{e}$ is a bounded linear map with $\|T^{e}\|=\|T\|_{e}$. Now consider $T^{ee} (= (T^e)^e) : M^{ee} \xrightarrow{} Y^{ee}$. Then $T^{ee}(\phi)=\phi \circ T^{e}.$ Thus for $g \in Y^{e}$,  $$\left(T^{ee}(\delta_{x}^{e})\right)(g)=(\delta_{x}^{e} \circ T^{e})(g)= \delta_{x}^{e}(g\circ T)= g(T(x)) = \delta_{(T(x))}^{e}(g),$$ 
    so that $T^{ee}(\delta_{x}^{e}) = \delta_{(T(x))}^{e}.$  Moreover, $\|T^{ee}\|=\|T^{e}\| = \|T\|_{e}.$
\end{defn}
\subsection{Position of $X^{\ast}$ in $X^{e}$}
Let $T \in X^{\ast} \subset X^{e}$. Then we have already proved that $\|T\|=\|T\|_{e}$. Therefore $X^{\ast}$ is a closed subspace of $(X^{e}, \|.\|_{e})$. We explicitly find out the description of this particular closed subspace.
\begin{lem} \label{lem beta}
	Let $X$ be a normed linear space. Then there exists a linear contraction $\beta_{X}^{e} : F_{e}(X) \xrightarrow{} X$.
\end{lem}
\begin{proof}
	Define $\beta_{X}^{e} : span\{\delta_{x}^{e} : x \in X\} \xrightarrow{} X$ given by $$\beta_{X}^{e}\big(\sum\limits_{i=1}^{n}\alpha_{i} \delta_{x_{i}}^{e}\big)= \sum\limits_{i=1}^{n}\alpha_{i} x_{i}.$$
	Then $\beta_{X}^{e}$ is a well-defined linear map. Also,
	\begin{eqnarray*}
		\big\|\sum\limits_{i=1}^{n} \alpha_{i} x_{i}\big\|&=& \sup\bigg\{ \big|\sum\limits_{i=1}^{n} \alpha_{i} x^{\ast}(x_{i})\big| : x^{\ast} \in X^{\ast}, \|x^{\ast}\| \leq 1\bigg\}\\
		&\leq&  \sup\bigg\{ \big|\sum\limits_{i=1}^{n} \alpha_{i} f(x_{i})\big| : f \in X^{e}, \|f\|_{e} \leq 1\bigg\}\\
		&=& \sup\bigg\{ \big|\Big(\sum\limits_{i=1}^{n} \alpha_{i} \delta_{x_{i}}^{e}\Big)(f)\big| : f \in X^{e}, \|f\|_{e} \leq 1\bigg\}\\
		&=& \big\|\sum\limits_{i=1}^{n} \alpha_{i} \delta_{x_{i}}^{e}\big\|.
	\end{eqnarray*}
	Thus $\beta_{X}^{e}$ is a bounded linear contraction with $\beta_{X}^{e} \circ \delta_{X}^{e} = Id_{X}$ and hence can be extended uniquely to $F_{e}(X)$.
\end{proof}
\begin{prop} \label{prop rb duals}
	Let $T \in X^{e}$. Then $T \in X^{\ast}$ if and only if $\mu(T)=0$ for all $\mu \in ker(\beta_{X}^{e})$.
\end{prop}
\begin{proof}
	Let $T \in X^{e}$ such that $\mu(T)=0$ for all $\mu \in ker(\beta_{X}^{e})$.\\
	\underline{claim:} $T \in X^{\ast}$.\\
	Suppose $r \in \mathbb{R}$ and $x \in X$. Now for $r\delta_{x}^{e}+ \delta_{(-rx)}^{e} \in Ker(\beta_{X}^{e})$, $(r\delta_{x}^{e}+ \delta_{(-rx)}^{e})(T)=0$ gives $-rT(x)=T(-rx)$.\\
	Again let $x_{1},x_{2} \in X$. Then $\delta_{(x_{1}+x_{2})}^{e}+ \delta_{(-(x_{1}+x_{2}))}^{e}, \ \delta_{x_{1}}^{e}+\delta_{
		x_{2}}^{e}+\delta_{(-(x_{1}+x_{2}))}^{e}\in ker(\beta_{X}^{e})$.\\
	Therefore, $(\delta_{(x_{1}+x_{2})}^{e}+ \delta_{(-(x_{1}+x_{2}))}^{e})(T)=0$ implies $T(-(x_{1}+x_{2}))= -T(x_{1}+x_{2})$ and $(\delta_{x_{1}}^{e}+\delta_{
		x_{2}}^{e}+\delta_{(-(x_{1}+x_{2}))}^{e})(T)=0$ gives $$T(x_{1})+T(x_{2})= -T(-(x_{1}+x_{2}))=T(x_{1}+x_{2}).$$
	Hence, the claim immediately follows.\\
	Conversely, let $T \in X^{\ast}$ and $\mu \in ker(\beta_{X}^{e})$. Therefore $\mu = \sum\limits_{n=1}^{\infty} a_{n} \delta_{x_{n}}^{e}$  
	for $(x_{n}) \subset X$ and $(a_{n}) \subset \mathbb{R}$ with $\sum\limits_{n=1}^{\infty} a_{n} x_{n} =0$ and hence $$\mu(T)= \sum\limits_{n=1}^{\infty} a_{n} \delta_{x_{n}}^{e} (T)= \sum\limits_{n=1}^{\infty} a_{n} T(x_{n})=T \big(\sum\limits_{n=1}^{\infty} a_{n} x_{n}\big) =T(0)=0.$$
\end{proof}
\begin{cor}
	By Lemma \ref{lem beta} and proposition \ref{prop rb duals}, we conclude that $X^{\ast}$ is precisely the annihilator of $\ker(\beta_{X}^{e})$ in $X^e$.
\end{cor}
\subsection{The (non-linear) operator algebra of extensively bounded mappings} 
Let $X, Y, Z$ be normed linear spaces. Then for $f, f_{1}, f_{2} \in E(X, Y)$ and $g, g_{1}, g_{2} \in E(Y,Z)$, we have  $g \circ f \in E(X,Z)$, $(g_{1}+g_{2}) \circ f = g_{1}\circ f + g_{2} \circ f$ but, in general, $g\circ (f_{1}+f_{2}) \neq g\circ f_{1}+ g \circ f_{2}$. However, when $g \in L(Y, Z) \subset E(Y, Z)$, we have $g\circ (f_{1}+f_{2}) = g\circ f_{1}+ g \circ f_{2}$. Further 
\begin{eqnarray*}
	\|g \circ f\|_{e} = \sup\limits_{x \neq 0} \frac{\|g(f(x))\|}{\|x\|} \leq \|g\|_{e} \sup\limits_{x \neq 0} \frac{\|f(x)\|}{\|x\|} = \|g\|_{e} \|f\|_{e}.
\end{eqnarray*} 
These properties exhibit an `operator ideal' like structure on the space of extensively bounded maps of Banach spaces with a pinch of non-linearity. A similar situation has been observed in the case of Lipschitz maps also. 

In this sense, $E(X, X)$ is a (non-linear) normed algebra with unity and therefore, it is a (non-linear) unital Banach algebra whenever $X$ is complete.

Let us continue with the above discussion. We observe that since $g \circ f \in E(X, Z)$, $\Phi_X^Z(g \circ f)$ makes sense. However, as $\Phi_Y^Z(g)$ can not composed with $\Phi_X^Y(f)$, $\Phi$ is not `multiplicative'.
We dilate Theorem \ref{lin} further to get a `multiplicative' linearization. 
 \begin{thm} \label{thm tilde}
    Let $(M,d)$ be a metric space and $Y$ be a normed linear space, and $T \in E(M,Y)$. Then there exists a unique bounded linear map $\widetilde{T} : F_{e}(M) \xrightarrow{} F_{e}(Y)$ such that $\widetilde{T} \circ \delta_{M}^{e} = \delta_{Y}^{e} \circ T$ with $\|\widetilde{T}\|=\|T\|_{e}$.
\end{thm}
\begin{proof}
    Let us first define the map 
    $$\widetilde{T} : span\left\{\delta_{x}^{e}: x \in M\right\} \to span\left\{\delta_{y}^{e}: y \in Y\right\}$$ 
    given by 
    $$\widetilde{T}\left(\sum\limits_{i=1}^{n} \alpha_{i} \delta_{x_{i}}^{e}\right) = \sum\limits_{i=1}^{n} \alpha_{i} \delta_{Tx_{i}}^{e}.$$
     Therefore by definition $\widetilde{T}$ is a linear map and $\widetilde{T} \circ \delta_{M}^{e} = \delta_{Y}^{e} \circ T$. Now for $x_{1},...,x_{n} \in M$ and the scalars $\alpha_{i}, i=1,2,...,n$
     \begin{eqnarray*}
         \left\|\sum\limits_{i=1}^{n} \alpha_{i} \delta_{x_{i}}^{e}\right\|&=& \sup\left\{ \left|\sum\limits_{i=1}^{n} \alpha_{i} f(x_{i})\right| : f \in M^{e}, \|f\|_{e} \leq 1\right\}\\
         &\geq& \frac{1}{\|T\|_{e}} \sup\left\{ \left|\sum\limits_{i=1}^{n} \alpha_{i} (g \circ T) (x_{i})\right| : g \in Y^{e}, \|g\|_{e} \leq 1\right\}\\
         &=&\frac{1}{\|T\|_{e}} \sup\left\{ \left|\sum\limits_{i=1}^{n} \alpha_{i} \delta_{(T(x_{i}))}^{e}(g)\right| : g \in Y^{e}, \|g\|_{e} \leq 1\right\}\\ 
         &=& \frac{1}{\|T\|_{e}} \left\|\sum\limits_{i=1}^{n} \alpha_{i} \delta_{Tx_{i}}^{e}\right\|.
     \end{eqnarray*}
     Thus 
     $$\left\Vert \widetilde{T}\left(\sum\limits_{i=1}^{n} \alpha_{i} \delta_{x_{i}}^{e}\right) \right\Vert = \left\|\sum\limits_{i=1}^{n} \alpha_{i} \delta_{Tx_{i}}^{e}\right\| \leq\|T\|_{e}\left\|\sum\limits_{i=1}^{n} \alpha_{i} \delta_{x_{i}}^{e}\right\|$$ 
     so that $\widetilde{T}$ is bounded. We extend $\widetilde{T}$ uniquely to $F_{e}(M)$ and retain the same symbol. Thus $\widetilde{T} \in L(F_e(M), F_e(Y))$ with $\|\widetilde{T}\| \leq \|T\|_{e}$. Again, 
     $$\|T\|_{e} =\sup\limits_{x \neq 0} \frac{\|Tx\|}{d(x,0)}= \sup\limits_{x \neq 0} \frac{\|\delta_{Tx}^{e}\|}{\|\delta_{x}^{e}\|}= \sup\limits_{x \neq 0} \frac{\|\widetilde{T}(\delta_{x}^{e})\|}{\|\delta_{x}^{e}\|} \leq \|\widetilde{T}\|$$
     so that $\Vert \widetilde{T} \Vert = \Vert T \Vert_e$. 
     
     Let $S \in L(F_e(M), F_e(Y))$ be such that $S \circ \delta_{M}^{e} = \delta_{Y}^{e} \circ T$. Then $\widetilde{T} \circ \delta_{M}^{e} = S \circ \delta_{M}^{e}$ so that $\widetilde{T} (\delta_{x}^{e}) = S (\delta_{x}^{e})$ for all $x \in M$. Since $\widetilde{T}$ and $S$ are linear and bounded, we get $\widetilde{T}(\gamma) = S(\gamma)$ for all $\gamma \in F_e(M)$. Thus $S = \widetilde{T}$. 
\end{proof}
\begin{lem} \label{cor tilde is multiplicative}
     Let $W, X, Y, Z$ be normed linear spaces and $U \in E(W, X)$, $T \in E(X, Y)$, $S \in E(Y, Z)$. Then $(S \circ T)^{\hat{}} = \hat{S} \circ \widetilde{T}$ and $\widetilde{S \circ T} = \widetilde{S} \circ \widetilde{T}$.
\end{lem}
\begin{proof}
    Applying Theorems \ref{lin} and \ref{thm tilde}, we get unique linear maps $\hat{S} \in L(F_e(Y), Z)$, $\widetilde{S} \in L(F_{e}(Y), F_{e}(Z))$ and $\widetilde{T} \in L(F_{e}(X), F_{e}(Y))$ such that $\hat{S} \circ \delta_Y^e = S$, $\widetilde{S} \circ \delta_{Y}^{e} = \delta_{Z}^{e} \circ S$ and $\widetilde{T} \circ \delta_{X}^{e} = \delta_{Y}^{e} \circ T$. Further  
    $$\hat{S} \circ \widetilde{T} \circ \delta_X^e = \hat{S} \circ \delta_Y^e \circ T = S \circ T = (S \circ T)^{\hat{}} \circ \delta_X^e$$ 
    and  
    \begin{eqnarray*} \label{eqn t1}
        (\widetilde{S} \circ \widetilde{T}) \circ \delta_{X}^{e} = (\widetilde{S} \circ \delta_{Y}^{e}) \circ T = \delta_{Z}^{e} \circ (S \circ T) = \widetilde{S \circ T} \circ \delta_{X}^{e}.
        \end{eqnarray*} 
    Thus $(S \circ T)^{\hat{}} = \hat{S} \circ \widetilde{T}$ and $\widetilde{S \circ T} = \widetilde{S} \circ \widetilde{T}$. 
\end{proof}
\begin{thm} \label{thm psi}
	Let $X$ be a Banach space. Then $E(X, X)$ is a Banach algebra and the map $\Psi_X^X : E(X, X) \xrightarrow{} L(F_{e}(X), F_{e}(X))$ defined by $\Psi_X^X(T)= \widetilde{T}$, is an injective, unital, norm preserving and multiplicative extensively bounded mapping. We simply write $\Psi$ when the spaces are understood clearly.
\end{thm}
\begin{proof}
	From the above discussion, we know that $E(X, X)$ is a non-linear Banach algebra such that $\|.\|_{e}$ is submultiplicative. By Theorem \ref{thm tilde} and the corollary \ref{cor tilde is multiplicative}, we observe that $\Psi$ is a norm preserving, multiplication preserving, extensively bounded map. Also 
	$$\delta_x^e = \delta_X^e \circ Id_X(x) = \Psi(Id_X) \circ \delta_X^e(x) = \Psi(Id_X)(\delta_x^e)$$
	for all $x \in X$. Thus $\Psi(Id_{X})= Id_{F_{e}(X)}$. and to prove the injectivity of $\Psi$, consider $T, S \in E(X, X)$ such that $\widetilde{T}= \widetilde{S}$. Therefore composing $\beta_{X}^{e}$ (defined in Proposition \ref{lem beta}) and $\delta_{X}^{e}$ from the left and right side respectively we get $ \beta_{X}^{e} \circ \widetilde{T} \circ \delta_{X}^{e}= \beta_{X}^{e} \circ \widetilde{S} \circ \delta_{X}^{e}. \ i.e. \ T=S$.
	This completes the proof.
\end{proof}
\begin{rem}
	Note that, though $L(F_{e}(X), F_{e}(X))$ is a unital Banach algebra in the usual sense, $\Psi (E(X, X))$ is not as $\Psi$ is not linear. In fact, $\Psi (E(X, X))$ is not even a (non-linear) Banach algebra like $E(X, X)$.
\end{rem}
\begin{prop} \label{r bw phi and psi}
    Let $X$, $Y$, $Z$ and $W$ be Banach spaces. Assume $ U \in E(X, Y)$, $T \in E(Y, Z)$ and $S \in E(Z, W)$. Then 
    $$\Phi_X^W(S \circ T \circ U) = \Phi_Z^W(S) \circ \Psi_Y^Z(T) \circ \Psi_X^Y(U).$$ 
    In particular, if $S \in L(Z, W)$, then 
    $$\Phi_X^W(S \circ T \circ U) = S \circ \Phi_Y^Z(T) \circ \Psi_X^Y(U).$$ 
    \end{prop}
    \begin{proof}
         Applying Theorem \ref{thm tilde}, we get unique $\Psi_X^Y(U) \in L(F_e(X), F_e(Y))$ and $\Psi_Y^Z(T) \in L(F_e(Y), F_e(Z))$ such that $\Psi_X^Y(U) \circ \delta_X^e = \delta_Y^e \circ U$ and $\Psi_Y^Z(T) \circ \delta_Y^e = \delta_Z^e \circ T$. Similarly, by Theorem \ref{lin}, there exist unique $\Phi_Y^Z(T) \in L(F_e(Y), Z)$ and $\Phi_Z^W(S) \in L(F_e(Z), W)$ such that $\Phi_Y^Z(T) \circ \delta_Y^e = T$ and $\Phi_Z^W(S) \circ \delta_Z^e = S$. Therefore, 
         \begin{eqnarray*}
             \Phi_Z^W(S) \circ \Psi_Y^Z(T) \circ \Psi_X^Y(U) \circ \delta_X^e &=&  \Phi_Z^W(S) \circ \Psi_Y^Z(T) \circ \delta_Y^e \circ U \\ 
             &=& \Phi_Z^W(S) \circ \delta_Z^e \circ T \circ U \\
             &=&S \circ T \circ U \\
             &=& \Phi_X^W(S \circ T \circ U) \circ \delta_{X}^e
         \end{eqnarray*}
         and
           \begin{eqnarray*}
    	S \circ \Phi_Y^Z(T) \circ \Psi_X^Y(U) \circ \delta_X^e &=& S \circ \Phi_Y^Z(T) \circ \delta_Y^e \circ U \\ 
    	&=& S \circ T \circ U \\ 
    	&=& \Phi_X^W(S \circ T \circ U) \circ \delta_X^e.
    \end{eqnarray*}
    Thus by Theorem \ref{lin}, we get the result.
 \end{proof}
 \begin{rem}
     The notions introduced in \ref{def e adjoint} can be rephrased in the following way using the maps $\Phi$ and $\Psi$ as defined in \ref{lin} and \ref{thm psi}, respectively. Let $T \in E(X, Y)$. Then $T^e \in L(Y^e, X^e)$ so that for $f \in Y^e$, we have 
     \begin{eqnarray*}
         \left(\Phi_X^{\mathbb{R}}\right)^{-1} \circ \Psi(T)^{\ast} \circ \Phi_Y^{\mathbb{R}}(f) &=&   \left(\Phi_X^{\mathbb{R}}\right)^{-1} \left(\Psi(T)^{\ast} \left(\Phi_Y^{\mathbb{R}}(f)\right)\right)\\
         &=& \Psi(T)^{\ast} \left(\Phi_Y^{\mathbb{R}}(f)\right) \circ \delta_X^e \\
         &=& \Phi_Y^{\mathbb{R}}(f) \circ \Psi(T) \circ \delta_X^e \\
         &=& \Phi_Y^{\mathbb{R}}(f) \circ \delta_Y^e \circ T \\
         &=& f \circ T = T^e (f),
     \end{eqnarray*}
     Thus $T^e =  \left(\Phi_X^{\mathbb{R}}\right)^{-1} \circ \Psi(T)^{\ast} \circ \Phi_Y^{\mathbb{R}}$. In a similar manner, we can show that $T^e \big|_{Y^{\ast}} = \left(\Phi_X^{\mathbb{R}}\right)^{-1} \circ \Phi(T)^{\ast}$.
 \end{rem}
\section{Operator ideals of extensively bounded maps} \label{sec 3}
As an application of the algebraic structure of extensively bounded mappings between Banach spaces, we introduce the notion of an extensively bounded operator ideal (or, $E$-operator ideal, in short), replicating the notion of  operator ideal of linear mappings as presented in \cite{TNOI} and \cite{OI}. We begin with the following notion. 
\begin{defn}
    For Banach spaces $X$ and $Y$, we say that $T \in E(X, Y)$ is of finite rank whenever $span(T(X))$ is a finite dimensional subspace of $Y$. The set of all finite rank extensively bounded maps from $X$ to $Y$ is denoted by $F_E(X, Y)$.
\end{defn}
\begin{defn}
    Let $E$ denote the class of extensively bounded mappings and let $L$ denote the class of bounded linear mappings between an arbitrary pair of Banach spaces.  An extensively bounded operator ideal (or, $E$-operator ideal, in short) $\mathcal{I}_{E}$ is a subclass of $E$ satisfying the following conditions: 
\begin{enumerate}
	\item $\mathcal{I}_{E}(X, Y)$ is a subspace of $E(X, Y)$.
	\item $F_E(X, Y) \subset \mathcal{I}_{E}(X, Y)$.
	\item For any $S \in E(Z, X)$, $T \in \mathcal{I}_{E}(X, Y)$ and $U \in L(Y, W)$, we have $UTS \in \mathcal{I}_{E}(Z, W)$. 
\end{enumerate}
	whenever $X$, $Y$, $Z$ and $W$ are Banach spaces.
\end{defn}
We remark that composing only with bounded linear maps on the left in stead of arbitrary extensively bounded maps a compromise to retain the linear structure in $\mathcal{I}_E$. A similar practice is followed in the case of Lipschitz operator ideals. 

Next, we describe our first (and the smallest) $E$-operator ideal. 
\begin{lem}\label{fd}
	Let $X$ and $Y$ be two normed linear spaces  and $T \in E(X,Y)$.  Then $T \in F_E(X, Y)$ if and only if $\Phi(T) \in F(F_e(X), Y)$. 
\end{lem}
\begin{proof}
	Let $T \in E(X, Y)$ so that $\Phi(T) \in L(F_e(X), Y)$. Since $F_e (X) = \overline{span \left\lbrace \delta_{x}^{e} : x \in X \right\rbrace}$, we get 
	$$\Phi(T)(F_e(X)) = \overline{span \left\lbrace \Phi(T) (\delta_{x}^{e}) : x \in X \right\rbrace} = \overline{span \left\lbrace T(x) : x \in X \right\rbrace}.$$ 
	Thus for $T \in F_E(X, Y)$, $\Phi(T)(F_e(X)) = span (T(X))$ is a finite dimensional subspace of $Y$, that is, $\Phi(T) \in F(F_e(X), Y)$. The proof of the converse can be traced back in a similar way. 
\end{proof}
\begin{prop}
    $F_{E}$ is the smallest operator ideal.
\end{prop}
\begin{proof}
    Let $S, T \in F_{E}(X, Y)$ and $\alpha \in \mathbb{K}$, then $$span((\alpha S + T)(X)) \subset span(S(X)) + span(T(X))$$ and thus $\alpha S + T \in F_{E}(X, Y).$ Again, if $S \in E(Z, X)$, $T \in F_{E}(X, Y)$ and  $U \in L(Y, W)$, then $span (UTS(X)) = U(span(T(S(X))))$. This gives the ideal property of $F_{E}(X, Y).$ Further we can easily check that $\|UTS\|_{e} \leq \|U\|\|T\|_{e}\|S\|_{e}.$ Now, by definition $F_{E}(X, Y) \subset \mathcal{I}_E(X, Y)$ whenever $\mathcal{I}_E$ is an $E$-operator ideal. 
\end{proof}
\begin{defn}
	Let $X$ and $Y$ be Banach spaces and $T \in E(X, Y)$. We say that $T$  is called an $E$-compact mapping, if the set $\left\{\frac{T(x)}{\|x\|}: x \neq 0\right\}$ is pre-compact subset of $Y$ and the set of all $E$-compact mappings is denoted by $K_{E}(X, Y)$. 
\end{defn}
\begin{lem} \label{lem coh}
    Let $X$ be a Banach space. Then the closed unit ball of $F_{e}(X)$ is the closed absolute convex hull of $\left\{\frac{\delta_{x}^{e}}{\|x\|}: x \in X\setminus\{0\}\right\}$.
\end{lem}
\begin{proof}
    Put $S = \left\{\frac{\delta_{x}^{e}}{\|x\|}: x \in X\setminus\{0\}\right\} \subset F_{e}(X)$. Then the polar of $S$ is 
    \begin{eqnarray*}
        S^{\circ} &=& \left\{ \xi \in F_{e}(X)^{\ast} : \left\vert \xi \left( \frac{\delta_{x}^{e}}{\Vert x \Vert} \right) \right\vert \leq 1 : x \in X \setminus \{ 0 \} \right\} \\
        &=& \left\{\Phi_X^{\mathbb{R}}(T) \in F_{e}(X)^{\ast} : \left\vert\Phi_X^{\mathbb{R}}(T)\left(\frac{\delta_{x}^{e}}{\|x\|}\right)\right\vert \leq 1 : T \in X^{e}, \  x \in X\setminus\{0\}\right\}\\
        &=& \left\{\Phi_X^{\mathbb{R}}(T) \in F_{e}(X)^{\ast} : \left\vert\frac{\delta_{x}^{e}(T)}{\|x\|}\right\vert \leq 1 : T \in X^{e}, \  x \in X\setminus\{0\}\right\}\\
        &=& \left\{\Phi_X^{\mathbb{R}}(T) \in F_{e}(X)^{\ast} : \left\vert\frac{Tx}{\|x\|}\right\vert \leq 1 : T \in X^{e}, \  x \in X\setminus\{0\}\right\}\\
        &=&\left\{\Phi_X^{\mathbb{R}}(T) \in F_{e}(X)^{\ast} : \|T\|_{e} \leq 1 : T \in X^{e}\right\}\\
        &=& \Phi_X^{\mathbb{R}}\left(B_{X^{e}}\right).
    \end{eqnarray*} 
    Again the pre-polar of $\Phi_X^{\mathbb{R}}\left(B_{X^{e}}\right)$ is 
    \begin{eqnarray*}
        \leftindex^{\circ}{\left(\Phi_X^{\mathbb{R}}\left(B_{X^{e}}\right)\right)} &=& \left\{\gamma \in F_{e}(X) : \left|\Phi_X^{\mathbb{R}}(T)(\gamma)\right|\leq 1 \mbox{ for \ all \ } T \in B_{X^{e}}\right\}\\
        &=& \left\{\gamma \in F_{e}(X) : \left|\gamma(T)\right|\leq 1 \mbox{ for \ all \ } T \in B_{X^{e}}\right\}\\
        &=&\left\{\gamma \in F_{e}(X) : \|\gamma\| \leq 1 \right\}\\
        &=& B_{F_{e}(X)}.
    \end{eqnarray*}
    Moreover as $B_{F_{e}(X)}$ is absolutely convex and closed, the bipolar theorem yeilds 
        $$B_{F_{e}(X)} = \leftindex ^{\circ}{\left(S^{\circ}\right)} = \overline{aco}\left\{\frac{\delta_{x}^{e}}{\|x\|}: x \in X\setminus\{0\}\right\}.$$
\end{proof}
\begin{lem} \label{lem equiv K_{E}() and K()}
     Let $X$ and $Y$ be two normed linear spaces and $T \in E(X,Y)$.  The following are equivalent:
    \begin{enumerate}
        \item $T \in K_{E}(X, Y)$.
        \item $\Phi(T) \in K(F_{e}(X),Y)$.
        \item $T^{e} \big|_{Y^{\ast}}$ is a linear compact operator from $Y^{\ast}$ to $X^{e}$.
    \end{enumerate}
\end{lem}
\begin{proof}
     By Lemma \ref{lem coh}, we get 
    \begin{eqnarray*}
    	\Phi(T)\left(B_{F_{e}(X)}\right) &=& \Phi(T)\left(\overline{aco}\left\{\frac{\delta_{x}^{e}}{\|x\|}: x \in X\setminus\{0\}\right\}\right) \\ 
    	&\subset& \overline{aco}\left\{\frac{Tx}{\|x\|} : 0 \neq x \in X\right\}
    \end{eqnarray*}
	for all $T \in E(X, Y)$. Let $T \in K_{E}(X, Y)$. Then $\left\{\frac{Tx}{\|x\|} : 0 \neq x \in X\right\}$ is precompact. Thus $aco\left\{\frac{Tx}{\|x\|} : 0 \neq x \in X\right\}$ is also precompact so that $\Phi(T) \in K_E(F_e(X), Y)$. 

    Since $\left\{\frac{Tx}{\|x\|} : 0 \neq x \in X\right\} \subset \Phi(T)\left(B_{F_{e}(X)}\right)$, we get $T \in K_E(X, Y)$ whenever $\Phi(T) \in K_E(F_e(X), Y)$. 
    This proves the equivalence of $(1)$ and $(2)$.
	
	Finally, since $T^{e} \big|_{Y^{\ast}} = \Phi^{-1} \circ \left(\Phi(T)\right)^{\ast}$ for any $T \in E(X, Y)$, $(2)$ is equivalent to $(3)$.
\end{proof}
\begin{thm}
$(K_E, \Vert\cdot\Vert_e)$ is a Banach operator ideal. 
\end{thm}
\begin{proof}
	Fix Banach spaces $X$ and $Y$. Let $S, T \in K_{E}(X, Y)$ and $(x_{n})$ be any sequence in $X\setminus \{0\}$. As $S \in K_{E}(X, Y)$, $\left(\frac{S(x_{{n}_{k}})}{\|x_{{n}_{k}}\|}\right)$ is Cauchy in $Y$ for some subsequence $(x_{n_k})$ of $(x_n)$. Again, as $T \in K_{E}(X, Y)$, we further get that  $\left(\frac{T(x_{{n}_{{k}_{l}}})}{\|x_{{n}_{{k}_{l}}}\|}\right)$ is Cauchy in $Y$ for some subsequence $(x_{n_{k_l}})$ of $(x_{n_k})$. Therefore, $\left(\frac{(S + T)(x_{{n}_{{k}_{l}}})}{\|x_{{n}_{{k}_{l}}}\|}\right)$ is a Cauchy subsequence of $\left(\frac{(S + T)(x_{n})}{\|x_{n}\|}\right)$ so that $S + T \in K_E(X, Y)$. Similarly, we can show that for $\alpha \in \mathbb{K}$, we have $\alpha S \in K_{E}(X, Y)$. Thus  $K_{E}(X, Y)$ is a subspace of $E(X, Y).$
 
	Next, we show that $\left(K_{E}(X, Y), \|.\|_{e}\right)$ is a Banach space. 

	Let $(T_{i})$ be a sequence in $K_{E}(X, Y)$ such that $\Vert T_i - T \Vert_e \to 0$ as $i \to \infty$ for some $T \in E(X, Y)$. We show that $T \in K_{E}(X, Y)$. 
	Let $(x_{n})$ be a sequence in $X\setminus \{0\}.$ Since $T_{1}$ is compact there exists a subsequence $(x_{{n}_{1,k}})_{k \in \mathbb{N}}$ such that $\left(\frac{T_{1}(x_{{n}_{1,k}})}{\|x_{{n}_{1,k}}\|}\right)_{k \in \mathbb{N}}$ is a Cauchy subsequence. Similarly $T_{2}$ is compact, therefore there exists $(x_{{n}_{2,k}})_{k \in \mathbb{N}}$ such that $\left(\frac{T_{2}(x_{{n}_{2,k}})}{\|x_{{n}_{2,k}}\|}\right)_{k \in \mathbb{N}}$ is also cauchy. Continuing inductively, we have for each $i \in \mathbb{N}, \ \ \left(\frac{T_{i}(x_{{n}_{i,k}})}{\|x_{{n}_{i,k}}\|}\right)_{k \in \mathbb{N}}$ converge in $Y.$ Consider the diagonal subsequence $x_{n_{k}}:= x_{n_{k,k}}$ for $k \in \mathbb{N}$ and fix $\epsilon>0.$ We prove that $\left(\frac{T(x_{n_{k}})}{\|x_{n_{k}}\|}\right)_{k \in \mathbb{N}}$ converges in $Y$ as well. Since $(T_{i})$ converges to $T,$ there exists $i_{0} \in \mathbb{N}$ such that $\|T-T_{i}\|_{e} < \frac{\epsilon}{3}$ for all $i \geq i_{0}.$ Again $\left(\frac{T_{i}(x_{n_{k}})}{\|x_{n_{k}}\|}\right)_{k \in \mathbb{N}}$ converges implies there exists a natural number $k_{0}$ such that $\left\|\frac{T_{i}(x_{n_{k}})}{\|x_{n_{k}}\|} - \frac{T_{i}(x_{n_{l}})}{\|x_{n_{l}}\|}\right\| < \frac{\epsilon}{3}$ for all $k,l \geq k_{0}.$ Now for all $k,l \geq \max\{k_{0}, i_{0}\}$
	\begin{eqnarray*}
		\left\|\frac{T(x_{n_{k}})}{\|x_{n_{k}}\|} - \frac{T(x_{n_{l}})}{\|x_{n_{l}}\|}\right\| &\leq& \left\|\frac{T(x_{n_{k}})}{\|x_{n_{k}}\|} - \frac{T_{i_{0}}(x_{n_{k}})}{\|x_{n_{k}}\|}\right\| + \left\|\frac{T_{i_{0}}(x_{n_{k}})}{\|x_{n_{k}}\|} - \frac{T_{i_{0}}(x_{n_{l}})}{\|x_{n_{l}}\|}\right\|\\
  &+&\left\|\frac{T_{i_{0}}(x_{n_{l}})}{\|x_{n_{l}}\|} - \frac{T(x_{n_{l}})}{\|x_{n_{l}}\|}\right\|\\
		&\leq& \|T-T_{i_{0}}\|_{e} + \left\|\frac{T_{i_{0}}(x_{n_{k}})}{\|x_{n_{k}}\|} - \frac{T_{i_{0}}(x_{n_{l}})}{\|x_{n_{l}}\|}\right\| + \|T_{i_{0}}-T\|_{e}\\
		&<& \epsilon.
	\end{eqnarray*}
	Hence $K_{E}(X, Y)$ is a Banach space.
    
    Now, we prove the E-operator ideal conditions. It follows from Lemmas \ref{fd} and \ref{lem equiv K_{E}() and K()} that $\Phi(F_E(X, Y)) = F(F_e(X), Y)$ and $\Phi(K_E(X, Y)) = K(F_e(X), Y)$. Since $F(F_e(X), Y) \subset K(F_e(X), Y)$, we deduce that $F_E(X, Y) \subset K_E(X, Y)$.
   
    Finally, let $W$, $X$, $Y$ and $Z$ be Banach spaces and assume that $ U \in E(Z, X)$, $T \in K_{E}(X, Y)$ and $S \in L(Y, W)$. Then by Theorem \ref{thm tilde}, there exists a unique $\tilde{U} \in L(F_e(Z), F_e(X))$ such that $\tilde{U} \circ \delta_Z^e = \delta_X^e \circ U$. Also by Lemma \ref{lem equiv K_{E}() and K()}, there exists a unique $\hat{T} \in K(F_e(X), Y)$ such that $\hat{T} \circ \delta_Y^e = T$. We write $\tilde{U} = \Psi_Z^X(U)$ and $\hat{T} = \Phi_X^Y(T)$.
    Thus from the proposition \ref{r bw phi and psi} we have $\Phi_Z^W(S \circ T \circ U) = S \circ \Phi_X^Y(T) \circ \Psi_Z^X(U)$. Since $K$ is a Banach operator ideal, we have 
    $$\Phi_Z^W(S \circ T \circ U) = S \circ \Phi_X^Y(T) \circ \Psi_Z^X(U)\in K(F_e(Z), W).$$
    Hence by Lemma \ref{lem equiv K_{E}() and K()}, $S \circ T \circ U \in K_E(Z, W)$.  
\end{proof}

\bibliographystyle{plain}

\bibliography{reference}

\end{document}